\documentclass[11pt]{article}

\usepackage[utf8]{inputenc}
\usepackage{lmodern}
\usepackage[normalem]{ulem} 
\usepackage{float}
\usepackage[margin=0.8in]{geometry}
\usepackage{enumitem}

\usepackage[utf8]{inputenc} 
\usepackage{tocloft}

\usepackage{graphicx, titlesec}
\usepackage{amscd,amsmath,amsthm,amssymb}
\usepackage{verbatim}
\usepackage[all]{xy}
\usepackage{color}
\usepackage{hyperref}
\usepackage{booktabs}
\usepackage{tikz-cd}
\usetikzlibrary{patterns}
\usepackage{ytableau}
\usepackage[norelsize,lined,ruled, linesnumbered ]{algorithm2e} 
\usepackage{booktabs}
\usepackage{cleveref}

\DeclareMathOperator{\Gr}{Gr}

\DeclareMathOperator{\codim}{codim}
\DeclareMathOperator{\rank}{rk}

\DeclareMathOperator{\conv}{conv}
\DeclareMathOperator{\CI}{CI}

\newcommand{\mD}{\mathcal{D}}
\newcommand{\mM}{\mathcal{M}}
\newcommand{\mP}{\mathcal{P}}
\newcommand{\mE}{\mathcal{E}}
\newcommand{\mB}{\mathcal{B}}
\newcommand{\mC}{\mathcal{C}}
\newcommand{\mF}{\mathcal{F}}
\newcommand{\mS}{\mathcal{S}}
\newcommand{\mR}{\mathcal{R}}
\newcommand{\mU}{\mathcal{U}}
\newcommand{\rk}{\textrm{rk }}
\newcommand{\ZZ}{\mathbb{Z}}
\newcommand{\RR}{\mathbb{R}}

\newcommand{\Sorted}{\textsf{Sorted}\xspace}

\definecolor{ao}{rgb}{0.0, 0.5, 0.0}
\definecolor{myred}{rgb}{0.81, 0.09, 0.13}

\newtheorem{theorem}{Theorem}[section]

\theoremstyle{definition}
\newtheorem{definition}[theorem]{Definition}
\newtheorem{notation}[theorem]{Notation}
\newtheorem{example}[theorem]{Example}
\newtheorem{remark}[theorem]{Remark}
\newtheorem{question}[theorem]{Question}

\theoremstyle{plain}
\newtheorem{lemma}[theorem]{Lemma}
\newtheorem{proposition}[theorem]{Proposition}
\newtheorem{corollary}[theorem]{Corollary}

\title{Combinatorics of Essential Sets for Positroids}
\author{Fatemeh Mohammadi and Francesca Zaffalon}
\date{}

\begin{document}

\maketitle
\begin{abstract}

\noindent 
Positroids {are a family} of matroids introduced by Postnikov in the study of non-negative Grassmannians. Postnikov identified several combinatorial objects in bijections with positroids, among which are bounded affine permutations. On the other hand, the notion of essential sets, introduced for permutations by Fulton, was used by Knutson in the study of the special family of interval rank positroids. 
We generalize Fulton's essential sets to bounded affine permutations. The bijection of the latter with positroids, allows the {study of the relationship between them}. From the point of view of positroids, essential sets are maximally dependent cyclic intervals. We define connected essential sets and prove that they give a facet description of the positroid polytope, as well as equations defining the positroid variety. We define a subset of essential sets, called core, which contains minimal rank conditions to uniquely recover a positroid. We provide an algorithm to retrieve the positroid satisfying the rank conditions in the core or any compatible rank condition on cyclic intervals.
\end{abstract}

\section{Introduction}

Matroids are combinatorial objects introduced to encode the concept of linear dependence arising from vector spaces, graphs, algebraic field extensions, and various other contexts. Positroids form a particularly well-behaved class of matroids.
On the combinatorial side, Postnikov has identified several families of objects in bijection with positroids, such as bounded affine permutations, plabic graphs, Grassmann necklaces, and Le diagrams~\cite{postnikov2006total}.~The realization space of a positroid inside the totally non-negative Grassmannian, called a positroid cell, is homeomorphic to an open ball. Positroid cells fit together to form a CW decomposition of the totally non-negative Grassmannian~\cite{postnikov2009matching}. Positroids have also been studied in relation to problems coming from quantum physics, as linear images of some positroid cells define a tiling of the amplituhedron, an object encoding the geometry of scattering amplitudes~\cite{Arkani-Hamed:2016byb, even2023cluster, mohammadi2021triangulations}.

\vspace{5pt}
Here, we introduce a new combinatorial description of positroids via \emph{ranked essential sets} and explain how to derive information about the positroid's structure from them. While many combinatorial objects studied in relation to positroids encode their data implicitly, ranked essential sets provide an intuitive understanding of the positroid's dependencies. These dependencies are succinctly encoded by relying on the cyclic order of the ground set. Additionally, ranked essential sets offer insights into the geometry of the positroid's realization space within the Grassmannian and the totally non-negative Grassmannian. As an example, the dimension {of positroid cells} can be explicitly computed in terms of the ranked essential sets.

\vspace{5pt}
Essential sets were initially introduced by Fulton \cite{fulton1992flags} as objects associated to permutations, to study the degeneracy loci of maps of flagged vector bundles. Their combinatorial characterization and properties were further investigated in \cite{eriksson1996combinatorics, eriksson1995size}. Their introduction within the context of positroids is due to Knutson \cite{knutson2010puzzles}, who used them in Fulton's original definition, for the subclass of interval rank positroids. These positroids are defined by rank conditions on intervals rather than cyclic intervals.

\vspace{5pt}
In this work, we extend the notion of essential sets from permutations to bounded affine permutations. Subsequently, we demonstrate how essential sets not only contain sufficient information to reconstruct the bounded affine permutation but also correspond to maximally dependent cyclic intervals of the positroid. We refer to the essential sets of a positroid, along with their rank, as the \emph{ranked essential family} of the positroid. The rank of a positroid on any cyclic interval can be computed in terms of the information given by the ranked essential family.
Furthermore, we define a subset of this family, known as \textit{core}, and prove that it gives minimal rank constraints necessary to uniquely define a positroid.

 \vspace{5pt}
\noindent\textbf{Structure of the paper.} In \Cref{section : preliminaries}, we introduce the language of matroids and positroids that we will need, as well as bounded affine permutations and their relation to positroids. In \Cref{section : ess set and permutation}, we define the diagram of a bounded affine permutation, whose top-right corners are defined to be the essential sets. In \Cref{subsection : rank and ess sets}, we prove an alternative definition via the rank function. This allows us to prove in \Cref{thm : rank cyclic intervals} that the information encoded in the ranked essential family is enough to uniquely describe a positroid. In \Cref{subsection : algorithm}, we present an algorithm, \Cref{alg : essential to permutation}, which constructs, whenever it exists, the bounded affine permutation of a positroid satisfying some rank conditions on cyclic interval. If the initial rank requirements form a ranked essential family, the algorithm is guaranteed to terminate and give the unique solution. The minimal set of rank conditions that will uniquely determine a positroid via \Cref{alg : essential to permutation} is introduced in \Cref{subsection : minimal rank conditions}. In \Cref{subsection : comb of essential sets}, we give an axiomatic description of positroids in terms of ranked essential families. In \Cref{section : positroid cells} we describe how ranked essential families characterize realization spaces of positroids. In the last section, \Cref{section : small rank}, we show that in the rank $2$ case, we retrieve the description provided in \cite{mohammadi2022computing}. 

\vspace{5pt}
\noindent \textbf{Acknowledgment.}~F.M. was partially supported by the grants G0F5921N (Odysseus programme) and G023721N from the Research Foundation - Flanders (FWO), the UiT Aurora project MASCOT and the grant iBOF/23/064 from the KU Leuven. F.Z. was supported by the FWO fundamental research fellowship (1189923N) and FWO long stay abroad grant (V414224N).

\section{Preliminaries}\label{section : preliminaries}
We begin this section by recalling the definitions of matroids and positroids. Specifically, we recall two among the various equivalent definitions of a matroid: one through its maximally independent sets, and the other through its closed sets~\cite{oxley2006matroid}.

\begin{definition}[Basis]
    A \emph{matroid} $\mM$ is a pair $(E, \mB)$ such that $E$ is a finite set, called \emph{ground set}, and $\mB$ is a collection of subsets of $E$, called \emph{bases}, satisfying:
    \begin{itemize}
       \item[(B1)] $\mB \neq \emptyset$;
        \item[(B2)] For every $A,B\in \mB$ and $a \in A\setminus B$, there exists an element $b\in B\setminus A$ such that $(A\setminus \{a\}) \cup \{b\}\in \mB$.
    \end{itemize}
    The size of any $B\in \mB$ is called \emph{rank} of the positroid.
\end{definition}

\begin{definition}[Flats]
    A matroid $\mM$ is a pair $(E,\mF)$ such that $E$ is a finite set, called \emph{ground set}, and $\mF$ is a collection of subsets of $E$, called \emph{closed sets} or \emph{flats}, satisfying:
    \begin{itemize}
        \item[(F1)] $E\in\mF$;
        \item[(F2)] If $X,Y\in \mF$, then $X\cap Y\in \mF$;
        \item[(F3)] If $X\in \mF$ and $x\in E$, there exists a unique $Y\in \mF$ such that $X\cup\{x\}\subseteq Y$ and there is no $Z\in \mF$ such that $X\cup\{x\}\subsetneq Z \subsetneq Y$.
    \end{itemize}
    We will denote the minimal flat containing a subset $S\subseteq E$, or \emph{closure} of $S$, by $\langle S\rangle$.
\end{definition}

\begin{notation}
    Throughout the paper, without loss of generality, we assume $E=\{1,\dots,n\}=[n]$. The family of cardinality $k$ subsets of $[n]$ is denoted by $\binom{[n]}{k}$. For $i\in \ZZ$, we denote by $i \mod n$ its representative in $[n]$.
    For any $i,j\in \mathbb{Z}$, 
    {we denote by $[i,j]$} the cyclic interval in $[n]$ from $i \mod n$, to $j \mod n$. For instance, if $n=8$, then $[7,10]=[7,2]=\{7,8,1,2\}$. We use $<$ to denote the normal ordering on $\mathbb{Z}$ and $<_i$ to denote the ordering on $[n]$ where $i$ is the smallest element, i.e. $i<_i i+1 <_i \dots <_i n <_i 1 <_i \dots <_i i-1$. {We extend this order to $\ZZ$ by defining, for $j,k\in \ZZ$, $j <_i k$ if and only if $j \mod n <_i k \mod n$.}
\end{notation}

\vspace{5pt}
Matroids arise as a generalization of the notion of linear dependency. If a matroid is associated to the dependency behavior of a finite set of vectors in a vector space, it is called realizable. Formally, {we have the following definition}.

\begin{definition}[Realizable matroid] \label{def : representable matroid}
    Let $\mathbb{K}$ be a field. A matroid $\mM = ([n], \mB)$ is said to be \emph{realizable over $\mathbb{K}$} if there exist {a positive integer $k$} and a $k \times n$ matrix $M$ with coefficients in $\mathbb{K}$ such that for every $I \in \binom{[n]}{k}$
    \[ \det(M_I) \neq 0 \quad \iff \quad I \in \mB,  \]
    where $M_I$ denotes the square $k \times k$ matrix whose columns are the columns indexed by $I$ in $M$. {The matrix $M$} is said to be a realizing matrix for $\mM$.
\end{definition}

A positroid is a matroid realizable over $\mathbb{R}$ with a positivity condition.
\begin{definition}[Positroid]
    A \emph{positroid} is a matroid $\mM$ which is realizable over $\mathbb{R}$ and such that there exists a realizing matrix $M$ for $\mM$ with non-negative maximal minors. In other words, for every $I\in \binom{[n]}{k}$, where $M_I$ is as defined in \Cref{def : representable matroid}, we have $\det(M_I) \geq 0$.
\end{definition}

Positroids, along with their combinatorial counterparts, have been subject to extensive research \cite{postnikov2006total}. {Moreover, they found applications in many problems coming from different branches of science~\cite{williamsICMtalk}. One important example is the study of the amplituhedron, an object encoding scattering amplitudes, quantities of interest in theoretical physics~\cite{arkani2014amplituhedron}.}
Among the combinatorial objects labeling positroids, bounded affine permutations are of particular interest to us.

\begin{definition}[Bounded affine permutation]
Consider a positive integer $n$. A \emph{bounded affine permutation of size $n$} is a bijection $\pi: \mathbb{Z} \to \mathbb{Z}$ satisfying the following conditions for all $i \in \mathbb{Z}$:
\[
\pi(i+n) = \pi(i) + n \quad \text{and} \quad i \leq \pi(i) \leq i + n.
\]
\end{definition}

Bounded affine permutations and positroids are in a one-to-one correspondence, as shown in \cite{postnikov2006total}. This correspondence is constructed as follows: given a positroid $\mathcal{P}$, we define the associated bounded affine permutation $\pi: \mathbb{Z} \to \mathbb{Z}$ by
\[ \pi(i) = \min\{j\geq i \mid i\in \langle [i+1,\dots, j]\rangle \} \quad \text{ for every } i \in \ZZ. \]
In particular, if $i$ is a loop of the positroid, then $\pi(i)=i$, and if $i$ is a coloop of the positroid, then $\pi(i)=i+n$. Moreover, information on the rank of a positroid on cyclic intervals can be extracted from its associated bounded affine permutation. {Indeed, by specializing \cite[Theorem 3]{rankfunction2020} to the case of cyclic intervals, we obtain the following result.}

\begin{proposition}\label{prop : properties positroid permutations}
    Let $\pi$ be a bounded affine permutation of size $n$ and let $\mP$ be the corresponding positroid.
\begin{enumerate}
    \item The rank of $\mP$ is equal to the number of $i\in [n]$ such that $\pi(i)>n$.
    \item The rank of a cyclic interval $[i,j]\subset [n]$ is given by the number of $\ell \in [i,j]$ such that $\pi(\ell)>j$.
\end{enumerate}
\end{proposition}

Another way to combinatorially verify that a matroid is a positroid is through its associated polytope.
\begin{definition}[Matroid polytope]
    Let $\mM=([n],\mB)$ be a matroid. The \emph{matroid polytope} of $\mM$ is
    \[ P_{\mM} = \conv(e_B \mid B\in \mB)\subseteq \RR^n \]
    where $e_B$ is the indicator vector of $B$, i.e. $e_B = \sum_{i\in B} e_i$, with $e_i$ vectors of the standard basis of $\RR^n$.
\end{definition}

Then the following characterization holds.
\begin{proposition}[\protect{\cite[Proposition~5.6]{ardila2016positroids}}] \label{prop : positroid polytope}
    A matroid $\mM$ is a positroid if and only if its matroid polytope $P_{\mM}$ can be described by the equality $x_1+\dots+x_n=k$ and by inequalities of the form
    \[ \sum_{\ell \in [i,j]} x_\ell \leq \rank([i,j]) \quad \text{ for } i,j\in [n].\]
\end{proposition}

\section{Essential sets and positroids} \label{section : ess set and permutation}

In this section, we will introduce the notion of essential sets of a bounded affine permutation and study their relations to the positroid associated to the same bounded affine permutation.

\subsection{Essential sets and bounded affine permutations}

\begin{notation}\label{notation}
   Consider a $k \times n$ array of squares in the plane. We will always consider the rows modulo $k$, i.e. like a matrix vertically infinitely repeating. The square at coordinates $(i,j)$ corresponds to the square contained in row $i$ and column $j$ of the array. Let $R_i$ be the $i$-th row of the array and $A_{(i,j)}$ be the antidiagonal containing the square $(i,j)$. The \emph{sub-antidiagonal} $\Delta_{(i,j)}$ of the square $(i,j)$ is the set of squares in the array contained in the antidiagonal starting at the square $(i+1,j-1)$ and ending in the first column, with rows considered modulo $k$. Formally, it is defined as:
    \[ \Delta_{(i,j)} = \{ (i+\ell \mod k, j-\ell) \mid 0< \ell < j \}. \]
   If $(i,j) = (i, 1)$, we have $\Delta_{(i,j)} = \emptyset$.
   For a square $(i,j)$, we denote by $T_{(i,j)}$ the set of squares contained in the triangle defined by the horizontal segment on the left of $(i,j)$ and the sub-antidiagonal $\Delta_{(i,j)}$. We denote by $P_{(i,j)}$ the set of squares to the right of the sub-antidiagonal $\Delta_{(i,j)}$ and contained between the horizontal segment starting to the right of $(i,j)$ and the row starting at the end of $\Delta_{(i,j)}$. See \Cref{fig : notation array}.
\end{notation}

\begin{figure}
    \centering
    \input{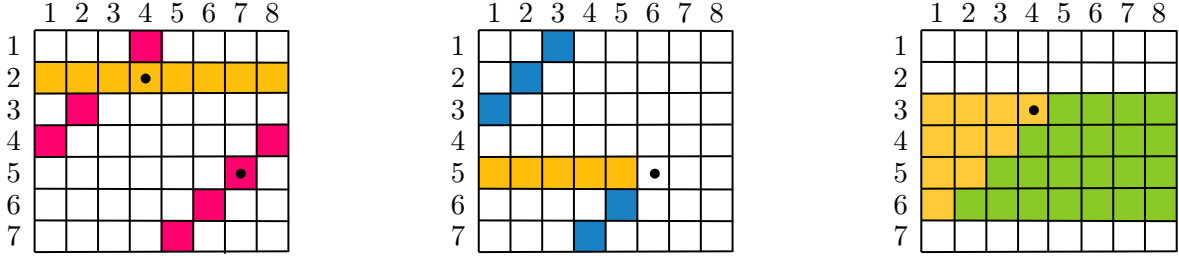}
    \caption{(Left) Yellow represents row $R_2$, and pink the  antidiagonal $A_{(5,7)}$ passing through $(5,7)$. (Center) Yellow indicates squares on the same row, strictly to the left of $(5,6)$, while blue signifies the sub-antidiagonal $\Delta_{(5,6)}$. (Right) Yellow depicts $T_{(3,4)}$, and green the set of squares $P_{(3,4)}$. See Notation~\ref{notation}.}
    \label{fig : notation array}
\end{figure}

Let $\pi$ be a bounded affine permutation of size $n$. We associate to it a dotted array as follows: consider an $n \times (n+1)$ array of squares in the plane. For each $i \in [n]$, place a dot inside the square $(i, \pi(i)-(i-1))$, while all other squares remain white. This is well-defined since $i\leq \pi(i) \leq i+n$. We denote by $D(\pi)$ the set of squares in the dotted array of $\pi$ containing a dot, i.e. $D(\pi) = \{ (i, \pi(i)-i+1) \mid i\in [n] \}$.

\begin{definition}[Diagram]
   The \emph{diagram of the bounded affine permutation} $\pi$ is constructed from the dotted array by shading. Specifically, for every dot in the array, the squares on the same row to the strict left of the dot and in its sub-antidiagonal are shaded. The diagram of $\pi$ consists of the remaining white squares.
\end{definition}

A white square is called a \emph{corner} if no {other} white square shares its top-right corner.
The \emph{essential family} $\mE(\pi)$ of $\pi$ is defined as the set of cyclic intervals $[i,j]$ for each white corner $(i,j-i+1)$ in the diagram of $\pi$. Each element in $\mE(\pi)$ is referred to as an \emph{essential set}. 
Equivalently:
\begin{align}
    \mE(\pi) &= \{ [i,j] \mid (i,j-i+1) \text{ is a corner of the diagram of } \pi \} \nonumber \\
    &= \{ [i,j] \mid i\in[n], i\leq j \leq i+n, \pi(i)\leq j, \, \pi^{-1}(j)\leq_i j, \, \pi(i-1) > j, \, j+1<_i \pi^{-1}(j+1) \}. \label{def : ess set}
\end{align}

\begin{example}\label{ex : diagram}
    Consider the bounded affine permutation of size $8$, defined by $\pi = (3 \; 4 \; 8 \; 7 \; 6 \; 9 \; 10 \; 13)$. The diagram of the permutation is depicted in \Cref{fig: example diagram} (left).
    \begin{figure}
        \centering
        \input{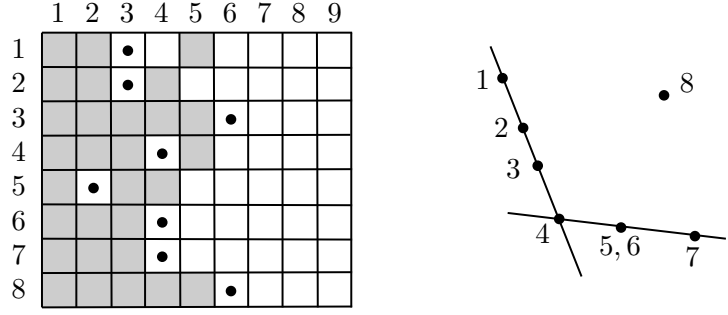}
        \caption{The diagram and the point-line configuration corresponding to $\pi$ from \Cref{ex : diagram}}
        \label{fig: example diagram}
    \end{figure}
    The corners of the diagram are $(1,4), (4,4),(5,2)$. Hence, the essential family is given by
    \[ \mE(\pi) = \{ [1,4], [4,7], [5,6] \}. \]
The positroid corresponding to $\pi$ has rank $3$ and can be represented by the point-line configuration in \Cref{fig: example diagram} (right). Note that the essential family captures the sets lying on a line and the parallel elements. However, to distinguish between them, we need to know their ranks.\end{example}

The properties of positroids that can be read through their associated bounded affine permutations can also be observed in their diagrams.
\begin{proposition}\label{prop : rank dots ess sets}
  Let $\pi$ be a bounded affine permutation of size $n$, and $\mathcal{P}$ its associated positroid.
  Then:
    \begin{enumerate}
        \item The rank of $\mP$ is equal to the number of dots in $P_{(1,n)}$.
        \item The rank of a cyclic interval $[i,j]\subseteq [n]$ is given by the number of dots in $P_{(i,j-i+1)}$.
    \end{enumerate}
\end{proposition}
\begin{proof}
    From \Cref{prop : properties positroid permutations}, the rank of $[i,j]$ is the number of $\ell\in[i,j]$ such that $\pi(\ell)>j$. This is exactly the number of dots in $P_{(i,j)}$.
\end{proof}

\begin{definition}[Ranked essential family]
    Let $\pi$ be a bounded affine permutation of order $n$. We define the \emph{ranked essential family} $\mE^{\mathrm{r}}(\pi)$ of $\pi$ to be the family
    \begin{align*}
        \mE^{\mathrm{r}}(\pi) = \left\{ (r,[i,j]) \mid [i,j]\in \mE(\pi)\cup\{[1,n]\}, \; r = \text{ number of dots in } P_{(i,j-i+1)} \right\}.
    \end{align*}
\end{definition}

\begin{example}
    Consider the permutation from \Cref{ex : diagram}. Its ranked essential family is
    \[ \mE^{\mathrm{r}}(\pi) = \{ (1,[5,6]), (2, [1,4]), (2, [4,7]), (3, [1,8]) \}. \]
\end{example}

\begin{example}\label{ex : ess set uniform matroid}
    Let $\mU_{k,n}$ be the uniform matroid of rank $k$ on $n$ elements. The associated bounded affine permutation is defined by $\pi_{k,n}(i)=i+k$, for every $i\in \ZZ$. It is easy to see that the set of corners of its diagram is empty, hence $\mE^{\mathrm{r}}(\pi_{k,n}) = \{(k,[1,n])\}$.
\end{example}

\subsection{Essential sets and rank function}\label{subsection : rank and ess sets}

The ranked essential family can be interpreted as the rank information of {maximally dependent} cyclic intervals. Indeed, the essential sets can be constructed from the rank function of the positroid as follows.

\begin{theorem}\label{thm : ess set and rank matrix}
    Let $\mP$ be a positroid on 
    $[n]$ of rank $k$ and $\mR$ the set of pairs $(r,[i,j])$, such that
    \begin{itemize}
        \item if $i\neq j$, then 
    $r=\rank([i,j]) = \rank([i+1,j]) = \rank([i,j-1]) = \rank([i-1,j])-1 = \rank([i,j+1])-1.$
    \item If $i=j$, then $0 = r = \rank([i,i]) = \rank([i-1,i])-1 = \rank([i,i+1])-1$.
    \end{itemize}
    Then $\mE^{\mathrm{r}}=\mR \cup \{(k,[1,n])\}$.
\end{theorem}
\begin{proof}
    Let $\pi$ be the bounded affine permutation associated to the positroid $\mP$. Let $(r,[i,j])\in \mR$. 
    Since $\rank([i,j])=\rank([i+1,j])$, we have $\pi(i)\leq j$. Similarly, since $\rank([i-1,j])=\rank([i,j])+1$ and $i-1\not\in \langle i,\dots, j\rangle$, we have $\pi(i-1)>j$.
    It remains to prove that $\pi^{-1}(j)\leq_i j$ and $\pi^{-1}(j+1)>_i j+1$. The condition $\rank([i,j])=\rank([i,j-1])$ is equivalent to the fact that the number of dots in $P_{(i,j-i+1)}$ and in $P_{(i,j-i)}$ is the same. Then either $\pi(j)=j$ or $\pi(j)>j$. In the latter case, $P_{(i,j-i+1)}$ contains a dot in the last row, which is not part of $P_{(i,j-i)}$. Hence there must be a dot in $P_{(i,j-i)}\setminus P_{(i,j-i+1)}$, i.e. $\pi^{-1}(j)<_i j$. It follows that $\pi^{-1}(j)\leq_i j$. Finally, $\rank([i,j])=\rank([i,j+1])-1$ implies that the number of dots in $P_{(i,j-i+2)}$ is one more than the number of dots in $P_{(i,j-i+1)}$. Note that $j+1$ is not a loop, hence $\pi(j+1)>j+1$, and $P_{(i,j-i+2)}$ contains a dot in the row $j+1$, which is not in $P_{(i,j-i+1)}$. For the equality to hold there must be no dot in $P_{(i,j-i+1)}\setminus P_{(i,j-i+2)}$, i.e. $\pi^{-1}(j+1)>_i j+1$. It follows that $[i,j]$ is an essential set of $\pi$.

\medskip

    Conversely, let $(r,[i,j])\in \mE^{\mathrm{r}}\setminus \{(k,[1,n])\}$. By definition, $r=\rank([i,j])$ is the number of dots in $P_{(i,j-i+1)}$. Moreover, $\rank([i+1,j])$ is equal to the number of dots in $P_{(i+1,j-i)}$, which is equal to the number of dots in $P_{(i,j-i+1)}$ since there is no dot in the first row of this region by definition of essential set. Therefore, $\rank([i+1,j])=r$. In the same way, since $\pi(i-1)>j$, the number of dots in $P_{(i-1,j-i+2)}$ is equal to one more than the dots in $P_{(i,j-i+1)}$, hence $\rank([i-1,j])=r+1$. Since $\pi^{-1}(j)\in[i,j]$, if $\pi^{-1}(j)=j$, then $j$ is a loop and $\rank([i,j-1])=\rank([i,j])$. If $\pi^{-1}(j)\in [i,j-1]$, then there is a dot in $P_{(i,j-i)}\setminus P_{(i,j-i+1)}$ as well as one dot in $P_{(i,j-i+1)}\setminus P_{(i,j-i)}$, hence $\rank([i,j-1])=\rank([i,j])$. Finally, since $\pi^{-1}(j+1)<i$, the squares in the left-most antidiagonal intersecting $P_{(i,j-i+1)}$ do not contain a dot. Hence, $\rank([i,j+1])$ is given by the number of dots in $P_{(i,j-i+1)}$ plus the number of dots contained in the last row of $P_{(i,j-i+2)}$, i.e. $\rank([i,j+1])=r+1$. It follows that $[i,j]$ satisfies the conditions in \cref{def : ess set}, i.e. $(r,[i,j])\in\mE^{\mathrm{r}}$.
\end{proof}

The rank of a positroid on any cyclic interval can be computed via its ranked essential family as follows. 

\begin{theorem}\label{thm : rank cyclic intervals}
     Let $\mP$ be a positroid and $\mE^{\mathrm{r}}$ its ranked essential family.~For any $[i,j]\subseteq [n]$, we have:
    \[ \rank([i,j]) = \min \{ r+|[i,j]\setminus I| \mid (r,I)\in \mE^{\mathrm{r}}\cup \{(0, \emptyset)\} \}. \]
\end{theorem}
\begin{proof}
    Denote by $m_{[i,j]}$ the quantity on the right-hand side of the equality. Note that $\rank([i,j]) \leq m_{[i,j]}$.
    
    Suppose $m_{[i,j]} = |[i,j]|$ and suppose for contradiction that $\rank([i,j]) < m_{i,j} = |[i,j]|$. Then there exists a circuit $C\subseteq [i,j]$ (a minimally dependent set). Let $a= \min_{<_i}C$ and $b= \max_{<_i}C$ and note that $[a,b]$ is such that $\rank([a,b]) = \rank([a+1,b])=\rank([a,b-1])$. Then there exists an essential set $(r,I)$ with $r = \rank([a,b]) \leq |[a,b]|-1$ and $[a,b]\subseteq I$. Note that $r+|[i,j]\setminus I| \leq |[a,b]|-1 + |[i,j]\setminus [a,b]| < |[i,j] = m_{[i,j]}$. Hence, we get a contradiction and $\rank([i,j]) = m_{[i,j]} = |[i,j]|$.

    Suppose that $m_{[i,j]} = r+ |[i,j]\setminus I|$ for some essential set $(r,I)$. Again, suppose for contradiction that $\rank([i,j]) < m_{i,j}$. Suppose $[i,j]\setminus I \neq \emptyset$. Then, there exists a circuit $C\subseteq [i,j]$ with $C\not\subseteq I$. Let $a= \min_{<_i}C$ and $b= \max_{<_i}C$ and note that $[a,b]$ is such that $\rank([a,b]) = \rank([a+1,b])=\rank([a,b-1])$ and at least one index between $a$ and $b$ is not in $I$. Then, there exists an essential set $(r',J)$ with $[a,b]\cup I\subseteq J$ and $r'\leq r+|(J\cap[i,j])\setminus I| -1$. Moreover, we have that:
    \[ r'+|[i,j]\setminus J| \leq r + |(J\cap[i,j])\setminus I|-1 + |[i,j]\setminus J| = r +|[i,j]\setminus I|-1< m_{[i,j]}. \]
    Finally, suppose $[i,j] \subseteq I$ and that $\rank([i,j]) < m_{i,j} = r$. In particular, $r \leq |[i,j]|$. If $r= |[i,j]|$, then the same reasoning as in the first case works. Hence, assume that $r< |[i,j]|$. Let $F$ be the union of circuits contained in $[i,j]$, then $F\neq \emptyset$, since $\rk([i,j])<|[i,j]|$. Let $[a,b]$ be the smallest cyclic interval containing $F$. Since $\rk([a,b])= \rk([a+1,b]) = \rk([a,b-1])$, there exists an essential set $(r',J)$ with $r' = \rk([a,b])<|[a,b]|$ and $J\supseteq [a,b]$. Then $\rk([i,j]) = r' +|[i,j]\setminus J| < r= m_{[i,j]}$, contradicting the hypothesis of minimality on $(r,I)$. 
\end{proof}

Since a positroid is uniquely defined by its rank on every cyclic interval, it follows that the ranked essential family of a positroid uniquely determines the positroid.

\begin{definition}[Connected essential set]
An essential set $(r,[i,j])\in \mE^{\mathrm{r}}$ is said to be \emph{connected} if there exist no pairwise disjoint sets $(r_1,[i_1,j_1]),\dots, (r_m,[i_m,j_m])\in \mE^{\mathrm{r}}$ such that
\[ r= r_1 + \cdots + r_m + \left| [i,j]\setminus \bigcup_{a=1}^m [i_a,j_a] \right| \quad\text{and}\quad [i,j]\supseteq \bigcup_{a=1}^m [i_a,j_a].\]
We denote the set of connected ranked essential sets by~$\mathcal{C}\mE^{\mathrm{r}}$. 
\end{definition}

\begin{remark}\label{rmk : rank from connected}
    By definition of connected essential set, it follows that we can rewrite the rank formula from \Cref{thm : rank cyclic intervals} as
    \[\rk([i,j]) = \min\left\{ \sum_{a=1}^m r_a + \biggl|[i,j]\setminus \bigcup_{a=1}^m [i_a,j_a]  \biggr| \text{ such that } \{ (r_a,[i_a,j_a])\}_{a=1,\ldots,m}\subseteq \mC\mE^{\mathrm{r}} \cup \{(0,\emptyset)\}  \right\} \]
    for every cyclic interval $[i,j]$.
    In particular, the rank conditions contained in connected essential intervals are sufficient to reconstruct the positroid. We will show in \S\ref{subsection : minimal rank conditions} that we can use an even smaller family of minimal rank conditions.
\end{remark}

The rank conditions given by connected essential sets define the facets of the corresponding positroid polytope.

\begin{corollary}
    Let $\mP$ be a positroid and let $\mE^{\mathrm{r}}$ be its ranked essential family. Then the positroid polytope of $\mP$ is given by the facets:
    \[ Q_\mP = \left\{ x\in \RR^n \mid 0 \leq x_i \leq 1 \; \forall i\in [n], \; \sum_{i\in[n]}x_i = k, \; \sum_{\ell\in [i,j]} x_\ell \leq r \quad \forall (r,[i,j])\in \mC\mE^{\mathrm{r}} \right\}. \]
\end{corollary}
\begin{proof}
    Let $P_\mP$ be the associated positroid polytope. Note that $x\in P_\mP$ satisfies all the inequalities defining $Q_{\mP}$, hence $P_\mP \subseteq Q_\mP$. Let $x \in Q_\mP$ and $[i,j]\subseteq [n]$. By \Cref{rmk : rank from connected}, we have that $\rank([i,j]) = \sum_{a=1}^m r_a + \bigl|[i,j]\setminus \bigcup_{a=1}^m [i_a,j_a]  \bigr|$ for some pairwise disjoint intervals $(r_1,[i_1,j_1]),\dots, (r_m,[i_m,j_m])\in \mE^{\mathrm{r}}$. Then 
    \begin{align*}
        \sum_{\ell\in [i,j]} x_\ell = \sum_{a=1}^m\sum_{\ell \in [i,j]\cap [i_a,j_a]} x_\ell + \sum_{\ell \in [i,j]\setminus \bigcup_{a=1}^m [i_a,j_a]} x_\ell \leq \sum_{a=1}^m r_a + \biggl|[i,j]\setminus \bigcup_{a=1}^m [i_a,j_a]  \biggr|.
    \end{align*}
    Equivalently, we have $\sum_{\ell \in [i,j]} x_\ell \leq \rank([i,j])$. Hence $x \in P_\mP$, as desired.
\end{proof}

\subsection{Properties of essential sets}

The ranked essential family contains the non-trivial rank information of a positroid. The elements of this set satisfy compatibility conditions as follows.

\begin{theorem}\label{thm : prop ess sets}
    Let $\mE^{\mathrm{r}}$ be the ranked essential family of a positroid $\mP$  of rank $k$ on the ground set $[n]$. Then, the following holds.
    \begin{itemize}
        \item[\rm{(E1)}] $(k, [1,n])\in \mE^{\mathrm{r}}$ and for each $(r,[i,j])\in \mE^{\mathrm{r}}$, we have 
        \[
        |[i,j]| > r \geq 0\quad\text{and}\quad|[n]\setminus [i,j]| \geq k-r>0.
        \]
        \item[\rm{(E2)}]\label{E2} For every distinct pair $(r_1,[i_1,j_1]), (r_2,[i_2,j_2]) \in \mE^{\mathrm{r}}\setminus \{(k, [1,n])\}$ such that $[i_1,j_1] \subsetneq [i_2,j_2]$, we have
        \[ 0 < r_2-r_1 < |[i_2,j_2]\setminus [i_1,j_1]|. \]
        \item[\rm{(E3)}]\label{E3} Let $(r_1,[i_1,j_1]),(r_2,[i_2,j_2])\in \mE^{\mathrm{r}}\setminus \{(k, [1,n])\}$ such that $i_2 <_{i_1} j_2$ and $j_1 <_{i_1} j_2$.
        \begin{itemize}
            \item[{\rm(1)}] If $[i_1,j_1]\cap [i_2,j_2] = \emptyset$, i.e. $j_1 <_{i_1} i_2$ and $j_2<_{i_2} i_2$, let $(r_3,[i_3,j_3])$ be a minimal essential set containing $[i_1,j_2]$ and $(r_4, [i_4,j_4])$ a maximal essential set contained in $[j_1+1,i_2-1]$. If there is no such $(r_4, [i_4,j_4])$, then let it be $(0,\emptyset)$.
Then
            \[ r_1 + r_2 \geq r_3 - r_4 - |[j_1+1,i_2-1]\setminus [i_4,j_4]|. \]
            \item[{\rm(2)}] If $[i_1,j_1]\cap [i_2,j_2]$ contains $[i_2,j_1]$, i.e. $i_2 <_{i_1} j_1$, let $(r_3,[i_3,j_3])$ be a minimal essential set containing $[i_1,j_1]\cup [i_2,j_2]$ and $(r_4, [i_4,j_4])$ a maximal essential set contained in $[i_2,j_1]$.  If there is no such $(r_4, [i_4,j_4])$, then let it be $(0,\emptyset)$.
Then
            \[ r_1 + r_2 \geq r_3 + r_4 + |[i_2,j_1]\setminus [i_4,j_4]|. \]
        \end{itemize}
    \end{itemize}
\end{theorem}
\begin{proof}
By construction of the essential family we have that $(k, [1,n])\in \mE^{\mathrm{r}}$. Moreover, if $(r,[i,j])\in \mE^{\mathrm{r}}\setminus\{(k, [1,n])\}$, then $0\leq r = \rank([i,j])= \rank([i+1,j])\leq |[i,j]|-1$.

    Let $(r_1,[i_1,j_1]), (r_2,[i_2,j_2]) \in \mE^{\mathrm{r}}$ be distinct ranked essential sets such that $[i_1,j_1] \subsetneq [i_2,j_2]$. Suppose $i_1-1\in[i_2,j_2]$ (similar argument works if $j_1+1\in[i_2,j_2]$), then $r_2 = \rank([i_2,j_2])\geq \rank([i_1-1,j_1])= \rank([i_1,j_1])+1>r_1$, hence $r_2-r_1>0$. Moreover, the following relation holds.
        \begin{equation}\label{eq : inequalities C2}
            r_2-r_1 \leq |\{\text{dots in } P_{(i_2,j_2-i_2+1)}\setminus P_{(i_1,j_1-i_1+1)}\}| \leq |[i_2,j_2]\setminus [i_1,j_1]|.
        \end{equation}
    Suppose $[i_2,j_2]\neq [1,n]$ is an essential set, then the dot in the antidiagonal through $(i_2,j_2-i_2+1)$ will be in row $h$ for $i_2<_{i_2} h<_{i_2} j_2$. If $h\in [i_1,j_1]$, then the number of dots in $P_{(i_2,j_2-i_2+1)}\setminus P_{(i_1,j_1-i_1+1)}$ is strictly larger than $r_2-r_1$. If $h\in [i_2,j_2]\setminus [i_1,j_1]$, then the number of dots in $P_{(i_2,j_2-i_2+1)}\setminus P_{(i_1,j_1-i_1+1)}$ is strictly smaller than $|[i_2,j_2]\setminus [i_1,j_1]|$. Hence, one of the two inequalities from \cref{eq : inequalities C2} is strict, i.e. $r_2-r_1 < |[i_2,j_2]\setminus [i_1,j_1]|$.

    Let $(r_1,[i_1,j_1]), (r_2,[i_2,j_2])$ be as in the assumption of {(E3)}. Let us prove some facts about the rank of the cyclic intervals involved in the statement.

    \medskip
    \noindent {\bf Claim 1.} There exists an essential set $(r,[a,b])$ containing $[i_1,j_1]\cup[i_2,j_2]$ of rank $\rank([i_1,j_2])$. Indeed, if $[i_1,j_2]=[1,n]$, then it is the essential set $(k,[1,n])$. Otherwise $\rank([i_1,j_2]) = \rank([i_1 +1, j_2]) = \rank([i_1,j_2-1])$ is either equal to $k$, hence $[a,b] = [1,n]$, or is strictly smaller than $k$ and so $[a,b]$, the maximal set containing $[i_1,j_2]$ with the same rank, will be an essential set.

       \medskip
 \noindent {\bf Claim 2.} Suppose $i_2 <_{i_1} j_1$, that is, $[i_1,j_1]$ and $[i_2,j_2]$ intersect in $[i_2,j_1]$. Let $(r,[a,b])$ be the maximal essential set contained in $[i_2,j_1]$, then $\rank([i_2,j_1]) = r + |[i_2,j_1]\setminus [a,b]|$. First note that by Claim 1 there is indeed a unique maximal essential set contained in $[i_2,j_1]$. Moreover, $r + |[i_2,j_1]\setminus [a,b]| \leq |[i_2,j_1]|$. Suppose there exists an essential set $(r',[c,d])$ such that $r' + |[i_2,j_1]\setminus [c,d]|< r + |[i_2,j_1]\setminus [a,b]|$. Then, by Claim 1 and property (E2), $[c,d]\not\subseteq [i_2,j_1]$. Moreover, $[i_2,j_1] \not\subseteq [c,d]$ since otherwise $\rank([i_2,j_1]) = r'$, and by \Cref{thm : ess set and rank matrix}, $[c,d]$ would be contained in both $[i_1,j_1]$ and $[i_2,j_2]$, hence $[i_2,j_1] = [c,d] = [a,b]$. Then suppose, without loss of generality, that $[i_2,j_1]\cap [c,d] = [c, j_1]$. Since $r' + |[i_2,j_1]\setminus [c,d]| < |[i,j]|$, we have that $|[c,j_1]|>r'$ and $[c,j_1]$ does not contain an essential set. Then $\rank([c,j_1]) = r'$ and $[c,d]$ is contained in $[i_1,j_1]\cap [i_2,j_2]$. Hence $\rank([i_2,j_1]) = r + |[i_2,j_1]\setminus [a,b]|$.

\medskip

    Let us now consider the two cases separately.
    \begin{itemize}
        \item[(1)] Suppose $[i_1,j_1]\cap [i_2, j_2] = \emptyset$. By Claim 1, $(r_3,[i_3,j_3])$ has the same rank as $[i_1,j_2]$. Moreover, by \Cref{thm : rank cyclic intervals}, $\rank([j_1+1,i_2-1]) \leq r_4+ |[j_1+1,i_2-1]\setminus [i_4,j_4]| $. Hence
        \[ r_1+r_2 = \rank([i_1,j_1]) + \rank([i_2,j_2]) \geq \rank([i_1,j_2]) - \rank([j_1+1,i_2-1]) \geq r_3 - r_4 - |[j_1+1,i_2-1]\setminus [i_4,j_4]|. \]
        \item[(2)] Suppose $[i_1,j_1]$ and $[i_2,j_2]$ intersect in $[i_2,j_1]$. Then, by Claim 1 and 
        Claim 2, it follows that
        \[ r_1+r_2 = \rank([i_1,j_1])+\rank([i_2,j_2]) \geq \rank([i_1,j_1]\cup[i_2,j_2]) + \rank([i_2,j_1]) = r_3 + r_4 + |[i_2,j_1]\setminus [i_4,j_4]| \]
        for $(r_3,[i_3,j_3]), (r_4,[i_4,j_4])$ as in the statement.
    \end{itemize}
\end{proof}

\begin{remark}
    The properties from \Cref{thm : prop ess sets}, can be interpreted as follow.
    \begin{itemize}
        \item[(E1)] If $(r,I)$ is a ranked essential set, then the rank is non-negative and the vectors in $I$ are dependent. 
        \item[(E2)] A ranked essential set $(r,I)$ contains all the vectors in the closure of $I$ in the same cyclic interval as $I$. Hence, if there is a ranked essential set $(r_2,I_2)$ containing $I$, its rank will be strictly larger and $I_2$ contains more dependency conditions than the ones in $(r,I)$.
        \item[(E3)] This condition corresponds to the sub-modular inequality for the rank function:
        \[ \rank(A)+\rank(B) \geq \rank(A\cup B) + \rank(A\cap B), \]
        in terms of ranked essential sets.
        
        Note that if two essential sets $(r_1,[i_1,j_1])$ and $(r_2,[i_2,j_2])$ intersect in two cyclic intervals, that is, $[i_1,j_1]\cap [i_2,j_2] = [i_1,j_2]\cup[i_2,j_1]$, by Theorem 25 of \cite{rankfunction2020}, the rank of the intersection is 
        \[ \rank([i_1,j_2]\cup[i_2,j_1]) = \min\{ r_1+r_2-k, r_3+|[i_1,j_2]\setminus[i_3,j_3]| +r_4 + |[i_2,j_1]\setminus [i_4,j_4]|\}, \]
        where $(r_3,[i_3,j_3])$ is the maximal essential set contained in $[i_1,j_2]$ and $(r_4,[i_4,j_4])$ is the maximal essential set contained in $[i_2,j_1]$. Hence, no additional condition is required for this case.
    \end{itemize}
\end{remark}
In \Cref{thm : chess theorem}, we will show that these properties are sufficient for a family of pairs consisting of a non-negative integer and a cyclic interval to form the ranked essential family of a positroid.

\subsection{Minimal rank conditions}\label{subsection : minimal rank conditions}

As seen in the definition of connected essential sets, the rank conditions within the ranked essential family of a positroid are not minimal. 
To restrict the ranked essential family to minimal conditions, we introduce a constant associated with each essential set. These constants intuitively count the number of additional dependent points in each essential set. We define them inductively as follows.

\begin{definition}[Excess] Let $\mE^{\mathrm{r}}$ be a ranked essential family. For each ranked essential set $(r,I) \in \mE^{\mathrm{r}}$,~let
\[ e_I = |I|-r-\sum_{J\in \mE, J\subsetneq I} e_J. \]
The quantity $e_I$ is called \emph{excess} of the essential set $I$.
\end{definition}

If the excess of a set vanishes, it indicates that there is no new information on the dependencies. Hence we give the following definition.

\begin{definition}[Core]
    Let $\mE^{\mathrm{r}}$ be a ranked essential family. The \emph{core} $\mC^{\mathrm{r}}$ of $\mE^{\mathrm{r}}$ is the subset given by essential sets with positive excess:
    \[ \mC^{\mathrm{r}} = \{ (r,I) \mid (r,I) \in \mE^{\mathrm{r}} \text{ and } e_I >0 \}. \]
\end{definition}

\begin{example}
    Consider the positroid from \cite[Example 1]{bonin2023characterization}. $\mP$ is defined to be the parallel connection of four copies of $\mU_{2,3}$ on $\{1,2,3\}, \{4,5,6\}, \{7,8,9\}, \{3,6,9\}$, {where the parallel connection of matroids is as defined in \cite[Proposition 7.1.18]{oxley2006matroid}}. Equivalently, $\mP$ is the positroid associated to the bounded affine permutation $\pi= (3 \; 10 \; 8 \; 6 \; 13 \; 11 \; 9 \; 16 \; 14)$. The ranked essential family is
    \[ \mE^{\mathrm{r}} =\{ (2,[1,3]),(2,[4,6]),(2, [7,9]),(4,[3,9]),(4,[6,3]),(4,[9,6]), (5,[1,9]) \}. \]
    This example illustrates an essential family where the smallest essential set containing the union of two essential intervals is strictly larger than the union. Specifically, the smallest essential set containing $[1,3]\cup [4,6]$ is $[9,6]\supsetneq [1,6]$.
    In this example, all rank conditions imposed by essential sets are necessary to define $\mP$. Moreover, the excess of each essential set is equal to $1$ and $\mE^{\mathrm{r}} = \mC^{\mathrm{r}}$.
    From the rank constraints on the essential sets, we can recover $\rk(\{3,6,9\}) = 2$ by applying \cite[Theorem 25]{rankfunction2020}.
\end{example}

\begin{remark}
    Let $\mE^{\mathrm{r}}$ be a ranked essential family. If an essential set $(r,[i,j])\in \mE^{\mathrm{r}}$ is disconnected, i.e. there exists pairwise disjoint essential sets  $(r_1,[i_1,j_1]),\ldots , (r_m,[i_m,j_m])\in \mE^{\mathrm{r}}$  such that $r=r_1 +\dots+ r_m + |[i,j]\setminus\bigcup_{a\in [m]} [i_a,j_a]| $ and $[i,j] \supseteq \bigcup_{a\in [m]} [i_a,j_a]$, then $(r,[i,j])$ is not in the core. Indeed
    \[ e_{[i,j]} \leq |[i,j]|-r - \sum_{a \in [m]} \left(|[i_a,j_a]|-r_a\right) = 0. \]
    However, in general the core $\mC^{\mathrm{r}}$ is strictly smaller than the family of ranked connected essential sets $\mC\mE^{\mathrm{r}}$. Consider for example the positroid with ranked essential family
    \[ \mE^{\mathrm{r}} = \{ (2,[1,3]),(2,[3,5]),(3,[1,5]),(4,[1,6]) \}. \]
    Then every essential set is connected but the core is $\mC^{\mathrm{r}} = \{ (2,[1,3]),(2,[3,5]),(4,[1,6]) \}$.
\end{remark}

\begin{question}
    In \cite{eriksson1996combinatorics}, Eriksson and Linusson introduced the core of for essential sets of permutations and showed that it generally has a much smaller size than the ranked essential family. What can be said about the size of the ranked essential family and of the core for bounded affine permutations?
\end{question}

\subsection{The retrieval algorithm} \label{subsection : algorithm}

Here, we answer the question of how to reconstruct the bounded affine permutation associated to a ranked essential family. The algorithm we provide to answer this question is robust. Indeed, the algorithm outputs a bounded affine permutation whenever the input contains rank conditions on cyclic intervals which are realizable by a positroid.

\begin{definition}[Proper dotting]
    Consider an $n\times (n+1)$ array of squares in the plane. A subset $D$ of squares in the array is said to be a \emph{proper dotting} if $D$ contains at most one square in each row of the array and at most one square in each antidiagonal of the array. A proper dotting is called \emph{maximal} if it contains exactly one square in each row of the array.  The \emph{rank} and the \emph{dependency} of a square $(i,j)$ with respect to a proper dotting $D$ are defined as follows:
    \[ r_D(i,j) = |D\cap P_{(i,j)}| \quad\text{and}\quad d_D(i,j) = |D \cap T_{(i,j)}|. \]
\end{definition}

\begin{remark}\label{rmk : dotting positroid}
    A maximal proper dotting of a $n\times (n+1)$ array corresponds uniquely to a bounded affine permutation $\pi$ of size $n$, given by $\pi(i) = j+i-1$ for $i\in [n]$, where $j\in [n+1]$ is such that $(i,j)\in D$. Hence, a maximal proper dotting defines a positroid.
\end{remark}

\noindent\textbf{Description of \Cref{alg : essential to permutation}.} 
The algorithm takes as input $E^{\mathrm{r}}$, a subset of the array with non-negative integer labels, containing the required rank conditions. Its objective is to identify, if feasible, a positroid satisfying such rank conditions. 
In the process, $\mathcal{S}$ represents the shaded squares and $D$ contains the position of the dots added to the array at any point during the algorithm. Additionally, the algorithm initializes $r$ to $0$.

Consider a labeled square $(r,(i,j)) \in E^{\mathrm{r}}$. Define $a_D(i,j) = j-d_D(i,j)-r$. Let $h\in[i,j]$ be the smallest row, according to the order $<_i$, not containing a dot. We consider two cases:

\begin{enumerate}
\item If there exists $\ell \in [h+1,j]$ such that $h$ is dependent on $h+1, \dots, \ell$, then add a dot in the smallest such $\ell$. Repeat this process until no further addition is possible.
\item If no such $\ell$ exists, then proceed to the next labeled square $(r, (i',j'))$ or increase the label by one. This case occurs if the number of elements in the interval corresponding to the square $(i,j)$ minus the dots already contained in $T_{(i,j)}$ is smaller or equal to $r$.
\end{enumerate}

Once we have applied the first part to each labeled square, we proceed to fill the remaining rows. 
Let $k$ be the maximum label, which is usually appearing as the label of $(1,n)$. For each row $h$ not containing a dot, add a dot in the minimal square with rank $k$, i.e. the minimal $\ell$ such that $\ell$ minus the number of dots contained in $T_{(h,\ell)}$ equals $k$.
 The algorithm is formalized in \Cref{alg : essential to permutation}.
\SetKwComment{Comment}{$\textcolor{gray}{\triangleleft}$ }{}
\NlSty{\texttt} 
\BlankLine
\begin{algorithm}[h!]
\caption{Construction of positroid from rank conditions on cyclic intervals}
\label{alg : essential to permutation}
\KwIn{$E^{\mathrm{r}}$ a set of pair of integer number and labels for a square in an $n \times (n+1)$ array.
\\
}
\BlankLine
\KwOut{A dotting $D$ of the array.}
\BlankLine
$D \leftarrow \emptyset, \hspace{0.5em} \mathcal{S} \leftarrow \emptyset, \hspace{0.5em} r\leftarrow 0, \hspace{0.5em} k\leftarrow \max\{ r \mid (r, (i,j))\in E^{\mathrm{r}}\}$\;
$E^{\mathrm{r}} \leftarrow \Sorted(E^{\mathrm{r}})$  \Comment*[r]{\textcolor{gray}{Sort $E^{\mathrm{r}}$ by label $r$ in ascending order}}
\ForEach{ $(r,(i,j)) \in E^{\mathrm{r}}$ } { 
$ a \leftarrow j-r-d_D(i,j)$\;
\While{$a>0$}{
$h \leftarrow \min_{<_i}\{ \alpha \mid R_\alpha \not\subseteq \mS \}$\;
$\ell \leftarrow \min\{ \beta \mid \beta-1 -d_D(h,\beta) =r \}$\;
$D \leftarrow D \cup \{(h,\ell)\}$\;
$\mS \leftarrow \mS \cup R_h \cup A_{(h,\ell)}$\;
$ b \leftarrow j-r-d_D(i,j)$\;
\eIf{$a = b$}{ \Return{\texttt{\emph{\color{myred}$E^{\mathrm{r}}$ is not valid}}}}
{$a \leftarrow b$; }
}
}
\ForEach{ $h$ such that $R_h \not\subseteq \mS$ }{
$\ell \leftarrow \min\{ \beta \mid \beta-1 -d_D(h,\beta) =k \}$\;
\eIf{$\ell >n+1$}{ \Return{\texttt{\emph{\color{myred}$E^{\mathrm{r}}$ is not valid}}}}{
$D \leftarrow D \cup \{(h,\ell)\}$\;
$\mS \leftarrow \mS \cup R_h \cup A_{(h,\ell)}$\;}
}
\If{ $\mS \neq [n]\times [n+1]$}{ \Return{\texttt{\emph{\color{myred}$E^{\mathrm{r}}$ is not valid}}}}
\ForEach{$(r,(i,j))\in E^{\mathrm{r}}$}{
\If{ $r_D(i,j) \neq r$ }{
{ \Return{\texttt{\emph{\color{myred}$E^{\mathrm{r}}$ is not valid}}}}
}
}
\Return{$D$}\;
\end{algorithm}

\newpage
\begin{remark}
    A retrieval algorithm for permutations and their essential sets is given in \cite{eriksson1995size}. However, the slightly different definition of essential sets introduced here, requires us to have a completely different approach in the retrieval algorithm. In particular, we use the bijection between positroids and bounded affine permutations to guide our approach.
\end{remark}

\noindent{\bf Implementation.} An implementation of the algorithm can be found at \begin{center} \href{https://github.com/frazaffa/essential-sets-positroid}{github.com/frazaffa/essential-sets-positroid} 
\end{center}
A visual implementation of the algorithm can be found at the following \href{https://sites.google.com/view/francescazaffalon/home/positroid?authuser=0}{website}. Here the required input is a list of triples, where each triple consists of the rank followed by the row and the column of the square in the array corresponding to the cyclic interval whose rank we want to assign.

\begin{example}
    Consider the set of labeled squares $E^{\mathrm{r}} = \{(1,(3,2)), (3, (1,5)) \}$. The steps of the algorithm applied to $E^{\mathrm{r}}$ can be seen in \Cref{fig : example alg}.
    \begin{figure}
        \centering
        \input{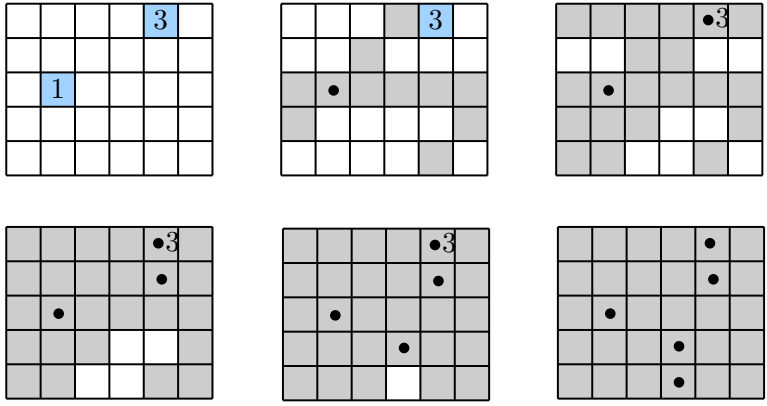}
        \caption{\Cref{alg : essential to permutation} applied to the labeled squares in blue in the first array.}
        \label{fig : example alg}
    \end{figure}
    Hence, the corresponding decorated permutation is $\pi = (5\; 6\; 4\; 7\; 8)$. 
\end{example}

The algorithm terminates if provided with compatible rank conditions and returns the dotting corresponding to a positroid. This positroid has a rank equal to the maximal label in the given squares and maximal among those satisfying the provided rank conditions.

\begin{theorem}
    Let $E^{\mathrm{r}}$ be a set of labeled squares from an $n\times (n+1)$ array and suppose we have a labeling of the square $(1,n)$. Then \Cref{alg : essential to permutation} terminates and outputs a proper dotting if and only if there exists a positroid $\mP$ satisfying the rank conditions corresponding to the input, i.e. for each $(r,(i,j))\in E^{\mathrm{r}}$, $\rank([i,j+i-1])=r$ by $\mP$ and such that the core of $\mP$ is contained in the rank information encoded in $E^{\mathrm{r}}$.
\end{theorem}
\begin{proof}
    Let $k$ denote the label of $(1,n)$. Note that since $k$ denotes the rank of the positroid, if $k$ is not the maximal label, then $E^{\mathrm{r}}$ is not valid.

    Note that the algorithm always terminates, either with a proper dotting or an error. Indeed, the errors in line {\tt $12$} and {\tt $21$} avoid entering an infinite loop and the error in line {\tt $28$} appears if the produced dotting is not a maximal proper dotting.
    
    \textbf{($\Leftarrow$)} Suppose the rank conditions in $E^{\mathrm{r}}$ can be satisfied by a positroid as in the statement. We want to prove that the output of the algorithm is a proper dotting. Let $\mathcal{D}$ be a maximal proper dotting such that $r_{\mD}(i,j)=r$ for every $(r,(i,j))\in E^{\mathrm{r}}$ and the core of the bounded affine permutation corresponding to $\mD$ is contained in the rank information given by $E^{\mathrm{r}}$. We want to prove that at every step of the algorithm, the dotting $D$ produced is a subset of $\mD$.

    In the first loop, for each $(r,(i,j))\in E^{\mathrm{r}}$ considered in increasing rank order, we add dots inside the triangle $T_{(i,j)}$, to restrict the rank of $[i,j+i-1]$ to be at maximum $r$. Let $D$ be the dotting at this step and suppose $D$ is a subset of $\mD$. We enter the loop if $j-d_D(i,j) > r$, that is whenever the number of dots in $T_{(i,j)}$ is not high enough. Since the quantity $j-r-d_D(i,j)$  is the excess of $[i,j+i-1]$, this is equivalent to the requirement that $(i,j)$ corresponds to a cyclic interval in the core of $\mD$. Let $h$ be the first empty row starting from $i$. Since $D$ can be extended to $\mD$ and $r_{\mD}(i,j)=r$, then $h \in [i,j+i-1]$. Let $\ell$ be such that $(h,\ell)\in \mD$. Note that $r_{\mD}(h, j+i-h)=r$. Indeed the dots that have been added so far $D$ in rows $i$ to $i+h-2$ lie inside $T_{(i,j)}$. Suppose not, let the first dot in rows $i$ to $i+h-2$ lying outside of $T_{(i,j)}$ be in square $(h_2,\ell_2)$, added while considering some square in $E^{\mathrm{r}}$ labeled by $r_2\leq r$. Then $i+j-h - d_D(h_2, i+j-h) < r_2$, hence $d_D(h_2,i+j-h) > i+j-h -r_2 \geq i+j-h -r$. Since all the rows between $i$ and $h_2$ have been filled by assumption with dots inside $T_{(i,j)}$, this would imply that $d_D(i,j) > h-i + i+j-h -r = j-r$, against the assumption that $a = j-r-d_D(i,j)>0$. It follows that $r_{\mD}(h,\ell)=r$ and $\ell\in T_{(i,j)}$. Moreover, $\ell' = \min\{\beta \mid \beta-1 -d_D(h,\beta) =r\}$ is such that $\ell=\ell'$ by rank maximality of $\mD$. Hence the algorithm does not produce an error and $D\cup \{(h,\ell')\} \subseteq \mD$.

    In the second for-loop, we add the dots outside of each $T_{(i,j)}$ considered in the first loop. Suppose $D$ is the dotting constructed so far and it is a subset of $\mD$ and let $h$ be the row we are considering. Let $\ell' = \min\{ \beta \mid \beta -1 -d_D(h,\beta)=k \}$, where $k$ is the label of the square $(1,n)$. Then, since for every $(r,(i,j))\in E^{\mathrm{r}}$, $d_D(i,j)\geq j-r$, $\ell'$ will lie outside of every triangle $T_{(i,j)}$. This guarantees that the rank of the square $(i,j)$ with respect to the dotting we are constructing is actually equal to $r$ for each $(r,(i,j))\in E^{\mathrm{r}}$. Let $\ell$ be such that $(h,\ell)\in \mD$. By rank maximality of $\mD$, $r_{\mD}(h,\ell)=k$ and $\ell \geq \ell'$. Moreover, suppose for contradiction that $\ell>\ell'$. Then since $r_{\mD}(h,\ell)= r_{\mD}(h,\ell')=k$, $\mD$ contains a dot in the antidiagonal containing $(h,\ell')$ in the rows between $h$ and $h+\ell'-1$. This is not possible by definition of $\ell'$. Hence $\ell=\ell'$, the dotting produced is still a subset of $\mD$ and the algorithm does not produce any error at this step.

    Since we assume $\mD$ to be a maximal proper dotting satisfying the rank conditions in $E^{\mathrm{r}}$ and we proved that the dotting $D$ constructed by the algorithm is equal to $\mD$, the last two checks will be satisfied by $D$, hence the algorithm will terminate and output $D$.
    
    \medskip
    \textbf{($\Rightarrow$)} Suppose now that the algorithm terminates and outputs a dotting, denoted as $D$. Note that $D$ must satisfy the rank conditions specified by $E^{\mathrm{r}}$, as the last step in the algorithm checks for this. Moreover, at every step, we choose a row $h$ that does not already contain a dot. As a consequence, at the end of the algorithm, all the boxes will be shaded if and only if no two dots lie in the same antidiagonal. Hence, $D$ is a proper dotting satisfying the rank conditions given in the input; that is, there exists a positroid satisfying $\rk([i,j+i-1])=r$ for every $(r,(i,j))\in E^{\mathrm{r}}$.
\end{proof}

In particular, it follows that the rank information contained in the core
is enough to uniquely recover the positroid.

\begin{corollary}\label{proposition : core}
    The rank conditions contained in the core of a positroid $\mP$ are minimal rank conditions defining $\mP$.
\end{corollary}

\section{Combinatorial characterization of essential sets} \label{subsection : comb of essential sets}
The goal of this section is to provide a combinatorial description of essential sets. We aim to introduce combinatorial conditions that guarantee a given set of labeled cyclic intervals can be viewed as the corners of the diagram of a bounded affine permutation.

The following result is a generalization of Theorem 4.1 from \cite{eriksson1996combinatorics} to the setting of bounded affine permutations and our definition of the rank of an essential set. We retain the name {``chess theorem''}, which stands for CHaracterization of ESsential Sets. The use of cyclic intervals and bounded affine permutations significantly alters the third condition in the theorem. The rank interpretation of essential sets allows for an elegant proof of the result. We divide its proof into Lemmas \ref{lemma : properties rank from e123} and \ref{lemma : bounded affine perm from e123}. \Cref{lemma : properties rank from e123} demonstrates that we can use a family $\mE^{\mathrm{r}}$ satisfying properties (E1), (E2), and (E3) to construct a function on cyclic intervals which shares many properties with the rank function of a positroid.

\begin{lemma}\label{lemma : properties rank from e123}
    Let $\mE^{\mathrm{r}}$ be a set of cyclic intervals labeled by a non-negative integer satisfying the properties {(E1), (E2), and (E3)} from \Cref{thm : prop ess sets}. Let $\CI_n$ denote the family of cyclic intervals of $[n]$ and $r: \CI_n \to \ZZ$ be the function defined by:
    \[ r([i,j]) = \min\{ |[i,j]|, r+|[i,j]\setminus I| \mid (r,I)\in \mE^{\mathrm{r}} \}. \]
    The function $r$ satisfies the following properties:
    \begin{itemize}
        \item[{\rm (r1)}] $r(I)=r$ for every $(r,I)\in \mE^{\mathrm{r}}$ and the minimum is uniquely achieved by $(r,I)$;
        \item[{\rm (r2)}] For every $[i,j]\in \CI_n$, $0\leq r([i,j]) \leq |[i,j]|$;
        \item[{\rm (r3)}] For every $[i,j]\in \CI_n\setminus [n]$, $r([i,j])\leq r([i-1,j]) \leq r([i,j])+1$ and $r([i,j])\leq r([i,j+1]) \leq r([i,j])+1$.
        \item[{\rm (r4)}] $[i,j]\in \CI_n$ is such that $r([i,j])=r([i+1,j])=r([i,j-1])=r([i-1,j])-1=r([i,j+1])-1$ if and only if $(r([i,j]),[i,j])\in \mE^{\mathrm{r}}$.
    \end{itemize}
\end{lemma}
\begin{proof}
 {\bf (r1):}   Let $(r,I)\in \mE^{\mathrm{r}}$. By property (E1), we have $r<|I|$. Suppose for contradiction that there exists $(r_2,I_2)\in \mE^{\mathrm{r}}$ such that $r_2+|I_2\setminus I| \leq r$. By (E2), $I_2$ neither contains nor is contained by $I$. Hence, let $(r_3,I_3)$ be a minimal set with the smallest label in $\mE^{\mathrm{r}}$ containing $I\cup I_2$, and let $(r_4,I_4)$ be a maximal set with the largest label contained in $I\cap I_2$. Suppose $I$ and $I_2$ intersect in a cyclic interval. By (E3), $r+r_2 \geq r_3 + r_4 + |(I\cap I_2)\setminus I_4|$. Moreover, $r_4 + |(I\cap I_2)\setminus I_4| + |I\setminus I_2| = r_4 + |I\setminus I_4| \geq r_2 + |I\setminus I_2|$, where the last equality holds by assuming that the minimum of the function $r$ is achieved by $(r_2,I_2)$. Then we can write $r+r_2+|I\setminus I_2| \geq r_3 + r_2 + |I\setminus I_2|$, hence $r \geq r_3$. Since $I\subsetneq I_3$, by  (E2), this is not possible.
Finally, suppose $I$ and $I_2$ intersect in two cyclic intervals. Then $r_2+|I\setminus I_2| = r_2+|[1,n]\setminus I_2| \geq r_2 +k-r_2+k >r$, where we used (E2). It follows that for every $(r_2,I_2)\in \mE^{\mathrm{r}}\setminus { (r,I)}$, $r_2+|I\setminus I_2|>r$, hence {(r1)} holds.

\medskip

 {\bf (r2):}      Property  {(r2)} holds by definition of the function $r$ and condition  {(E1)} on the family $\mE^{\mathrm{r}}$.

\medskip

  {\bf (r3):}     Let $[i,j]$ be a cyclic interval. We will prove  {(r3)} by considering some cases. 
  \begin{itemize}
      \item If $r([i,j])$ and $r([i-1,j])$ are given by the cardinality of the sets,  {(r3)} holds trivially. 
  \item Suppose $r([i,j]) = |[i,j]|$ and $r([i-1,j]) = r + |[i-1,j]\setminus I|$ for some $(r,I)\in \mE^{\mathrm{r}}$. Then $i-1 \in I$ and $r+|[i,j]\setminus I| \geq |[i,j]|$, hence $r \geq |I\cap [i,j]|$. It follows that $r([i-1,j]) = r + |[i-1,j]\setminus I| \geq |I\cap [i,j]| + |[i-1,j]\setminus I| = |[i,j]| = r([i,j])$. Moreover $r([i-1,j]) \leq |[i-1,j]| = r([i,j])+1$. 
  \item Finally, suppose $r([i,j]) = r + |[i,j]\setminus I|$ for some essential set $(r,I)$. Then clearly $r([i-1,j]) \leq r([i,j])+1$. If $r([i-1,j]) = |[i-1,j]|$ then $r([i-1,j]) \geq r([i,j])+1$. If $r([i-1,j]) = r' + |[i-1,j]\setminus I'|$ for some essential set $(r',I')$ then $r + |[i,j]|\setminus I| \leq r' + |[i,j]\setminus I'| \leq r([i-1,j])$. 
  \end{itemize}
  In an analogous way it is possible to show that $r([i,j])\leq r([i,j+1]) \leq r([i,j])+1$, hence  {(r3)} holds. 

\medskip
  {\bf (r4):}     Suppose $[i,j]\subsetneq [n]$ is such that $r([i,j])=r([i+1,j])=r([i,j-1])=r([i-1,j])-1=r([i,j+1])-1$ holds. Let $(r,I)\in \mE^{\mathrm{r}}$ such that $r([i,j])=r+|[i,j]\setminus I|$. Note that by assumption, $i-1,j+1 \not\in I$, hence $I\subseteq [i,j]$. Suppose $I\neq [i,j]$ and assume that $i\not\in I$. Then $r([i+1,j])\leq r + |[i+1,j]\setminus I|< r+|[i,j]\setminus I|=r([i,j])$, a contradiction. 
  In the same way we need to have that $j\in I$. Since $I$ is a cyclic interval, it follows that $I=[i,j]$. By  {(r1)}, $r([i,j])=r$, hence $(r([i,j]),[i,j])\in \mE^{\mathrm{r}}$. 

   Suppose now $(r,[i,j])\in \mE^{\mathrm{r}}$. By  {(r1)}, $r([i,j]) = r$ and the minimum is uniquely achieved by $(r,I)$. Hence, for every $(r_2,I_2)\in\mE^{\mathrm{r}}$, $r_2+|[i,j]\setminus I_2|\geq r+1$ and $r([i+1,j])\geq r$, $r([i,j-1])\geq r$. Since they are both contained in $[i,j]$, it follows that $r([i,j])=r([i+1,j])=r([i,j-1])=r$. Moreover, $r_2+|[i-1,j]\setminus I_2|\geq r+1$ and $r_2+|[i,j+1]\setminus I_2|\geq r+1$, and $r+|[i-1,j]\setminus [i,j]|=r+1=r+|[i,j+1]\setminus [i,j]|$, hence we can conclude that {(r4)} holds.
\end{proof}

The following result shows that we can use the function defined in \Cref{lemma : properties rank from e123} to construct a bounded affine permutation, which is a positroid. 

\begin{lemma}\label{lemma : bounded affine perm from e123}
    Let $\mE^{\mathrm{r}}$ be a family of labeled cyclic intervals satisfying  {(E1), (E2), and (E3)} from \Cref{thm : prop ess sets}. 
    Let $r: \CI_n \to \ZZ$ be the function defined in \Cref{lemma : properties rank from e123} and let $\pi:\ZZ \to \ZZ$ be defined as
    \[ \pi(i)=\min\{ j\geq i \mid r([i,j]) = r([i+1,j]) \}, \]
    where $[i,j] =\{i \mod n, \dots, j \mod n \}$ for every $i\leq j$. Then $\pi$ is a bounded affine permutation.
\end{lemma}
\begin{proof}
  The function $\pi$ satisfies the properties such that for every $i\in \ZZ$, $i\leq \pi(i)\leq i+n$ and $\pi(i+kn)=\pi(i)+kn$ for every $k \in \ZZ$ by construction. 
  Hence, it remains to prove that $\pi$ is a bijection.

    Suppose $i,j\in\ZZ$, with $i<j$, are such that $\pi(i)=\pi(j)=\ell$. Then $r([j,\ell])=r([j+1,\ell]) > r([j+1,\ell-1])$. By property  {(r3)} of the function $r$, $r([j+1,\ell-1])= r([j,\ell])-1$. We have that $r([j,\ell])\geq r([j,\ell-1])\geq r([j+1,\ell-1])= r([j,\ell])-1$. Suppose the first inequality is strict, then $r([j,\ell-1])= r([j+1,\ell-1])$, contradicting the minimality assumption on $\ell$. Hence $r([j,\ell])= r([j,\ell-1])$. In the same way, $r([i,\ell])=r([i,\ell-1])$. Hence there are $(r_i,I),(r_j,J) \in \mE^{\mathrm{r}}$ such that $[i,\ell]\subseteq I$, $[j,\ell]\subseteq J$ and $r_i=r([i,\ell])$, $r_j = r([j,\ell])$. By definition of $\pi(i)=\ell$, it follows that $r([i+1,\ell-1])= r_i-1$. Suppose $r([i+1,\ell-1]) = r_2+ |[i+1,\ell-1]\setminus I_2|$ for some $(r_2,I_2)\in \mE^{\mathrm{r}}$, then $I_2$ is contained in $[i+1,\ell-1]$ since if I add $i$ or $\ell$, the value of $r$ strictly increases. Then, $r_2+ |[i+1,\ell-1]\setminus I_2| \geq r_2+|I \setminus I_2| \geq r$ by property  {(E2)}, contradicting the assumption. Finally, note that $|[i+1,\ell-1]| \geq |[i+1,\ell-1]| + r' - |[j,\ell-1]| \geq r' +|[i+1,\ell-1]\setminus J| = r' +|[i+1,\ell]\setminus J| \geq r$. Hence, we get a contradiction, and the function $\pi$ is injective. Moreover, since $\pi$ is bounded and periodic, it follows that $\pi$ is a bijection.
\end{proof}

Finally, by combining the construction from Lemmas~\ref{lemma : properties rank from e123} and \ref{lemma : bounded affine perm from e123}, we obtain the characterization of essential sets of a positroid. This provides an axiomatic description of positroids from this perspective.

\begin{theorem}[Chess theorem]\label{thm : chess theorem}
    Let $\mE^{\mathrm{r}}$ be a set of cyclic intervals labeled by a non-negative integer. Then $\mE^{\mathrm{r}}$ is the ranked essential family of a positroid if and only if $\mE^{\mathrm{r}}$ satisfies the properties  {(E1), (E2), and (E3)} from \Cref{thm : prop ess sets}.
\end{theorem}
\begin{proof}
    If $\mE^{\mathrm{r}}$ is the ranked essential family of a positroid, then  {(E1), (E2), and (E3)} hold as shown in \Cref{thm : prop ess sets}.

    Suppose now $\mE^{\mathrm{r}}$ is family of labeled cyclic intervals satisfying  {(E1), (E2), and (E3)}. Let  $r: \CI_n \to \ZZ$ be the function defined in \Cref{lemma : properties rank from e123} and let $\pi:\ZZ \to \ZZ$ be the bounded affine permutation constructed in \Cref{lemma : bounded affine perm from e123}. We want to prove that the positroid $\mP$ associated to $\pi$ is such that $\rk([i,j])=r([i,j])$ for each cyclic interval $[i,j]$. Then, by \Cref{thm : ess set and rank matrix} and property  {(r4)} from \Cref{lemma : properties rank from e123}, it follows that $\mE^{\mathrm{r}}$ is the ranked essential family of the positroid $\mP$.
    We will prove by induction on the size of the cyclic interval that the rank `$\rk$' and `$r$' agree on each cyclic interval. For $i\in [n]$, we have
    \begin{align*}
        \rk([i,i])= \begin{cases}
            0 \quad & \text{if } \pi(i)=i \quad \iff \quad r([i,i])=0\\
            1 \quad & \text{if } \pi(i)>i \quad \iff \quad r([i,i])=1.
        \end{cases}
    \end{align*}
   Suppose now that $\rk$ and $r$ agree on each cyclic interval on size smaller than $m$ and suppose by contradiction $[i,j]$ is such that $|[i,j]|=m$ and $r([i,j])\neq \rk([i,j])$. We will consider the following cases:

    \medskip
      {\bf Case 1.} If $r([i,j])<\rk([i,j])$, since both $r$ and $\rk$ are satisfying the property  {(r3)} from \Cref{lemma : properties rank from e123}, $r([i,j])=r([i+1,j])$, hence $\pi(i)\leq j$. Moreover, $\rk([i,j])=\rk([i+1,j])+1$, i.e. $i$ is independent of $i=1,\dots, j$, hence $\pi(i)>j$ and we get a contradiction. 
     \medskip
    
     {\bf Case 2.} If $r([i,j])>\rk([i,j])$, since $\rk([i,j])=\rk([i+1,j])$, $i$ is dependent on $i+1,\dots,j$, hence $\pi(i)< j$. Again, consider the two distinct cases.
    Suppose then that $r([i,j])>\rk([i,j])$. Note that since $\rk([i,j])=\rk([i+1,j])$, $i$ is dependent on $i+1,\dots,j$, hence $\pi(i)< j$. 
    \medskip

    {\bf Case 2.1.} Suppose $\pi(j) = j$, then $r([j,j])=0$ hence there is $(0,I)\in \mE^{\mathrm{r}}$ such that $j\in I$. Suppose $r([i,j-1])=r'+|[i,j-1]\setminus I'|$ for some $(r',I')\in \mE^{\mathrm{r}}\cup \{(0,\emptyset)\}$. First assume that $I\cap I'\neq \emptyset$. Let $(r_3,I_3)$ be the minimal element in $\mE^{\mathrm{r}}$ containing $I\cup I'$ and let $(r_4,I_4)$ be the maximal element from $\mE^{\mathrm{r}}$ contained in $I\cap I'$. Then by  {(E3.2)} $0+r' \geq r_3+r_4+ |(I\cap I')\setminus I_4|$. Then $r' +|[i,j-1]\setminus I'| \geq r_3+r_4+ |(I\cap I')\setminus I_4|+|[i,j-1]\setminus I'| \geq r_3 + |[i,j-1]\setminus I_3|$. Hence $r([i,j-1])= r_3 + |[i,j-1]\setminus I_3| = r([i,j])$, where the last equality holds since $j\in I_3$. Assume now that $I\cap I'\neq \emptyset$. Similarly, let $(r_3,I_3)$ be the minimal element in $\mE^{\mathrm{r}}$ containing $I\cup I'$ and let $(r_4,I_4)$ be the maximal element from $\mE^{\mathrm{r}}$ contained in the minimal cyclic interval containing $I\cup I'$ in $[i,j]$ $C$, minus $I\cup I'$. Then by  {(E3.1)}, $0 + r' \geq r_3 - r_4- |C\setminus I_4|$. Then $r' +|[i,j-1]\setminus I'| \geq r_3 - r_4- |C\setminus I_4|+|[i,j-1]\setminus I'| \geq r_3 + |[i,j-1]\setminus I_3|$ and as before we get that $r([i,j])= r([i,j-1])$ and we get that $r$ and $\rank$ agree on $[i,j]$.
    \medskip

    {\bf Case 2.2.} Suppose $\pi(j)>j$. Since $\rk([i,j-1]) = \rk([i,j])$, by \Cref{prop : properties positroid permutations},  there is some $h \in [i+1,j-1]$ such that $\pi(h) = j$, i.e. $r([h,j]) = r([h+1,j]) = r([h, j-1])$. By  {(r4)}, there exists some $(r,I)$ such that $[h,j]\subseteq I$ and $r=r([h,j])=\rk([h,j])$. Let $(r',I')\in \mE^{\mathrm{r}}\cup \{(0,\emptyset)\}$ be such that $r([i,j-1]) = r' + |[i,j-1]\setminus I'|$. Suppose $I\cap I'\neq \emptyset$, the other case will be equivalent. Let $(r_3,I_3)$ be the minimal element in $\mE^{\mathrm{r}}$ containing $I\cup I'$ and let $(r_4,I_4)$ be the maximal element from $\mE^{\mathrm{r}}$ contained in $I\cap I'$. Then by  {(E3)}, $r' + |[i,j-1]\setminus I'| = r' + |[i,j-1]\setminus (I\cup I')| + |I\setminus I'| \geq r' + |[i,j-1]\setminus I_3| +|I\setminus I'| \geq r+r' - |I\cap I'| + |[i,j-1]\setminus I_3| \geq r_3 + |[i,j-1]\setminus I_3| + r_4 + |(I\cap I')\setminus I_4| - |I\cap I'| \geq r_3 + |[i,j-1]\setminus I_3|$. Hence $r([i,j-1]) = r_3 + |[i,j-1]\setminus I_3| = r_3 + |[i,j]\setminus I_3| = r([i,j])$, since $j\in I_3$. 
\end{proof}

\begin{remark}
    If $\mE^{\mathrm{r}}$ is a family of labeled cyclic interval satisfying  {(E1), (E2), and (E3)} we can recover the bounded affine permutation of the positroid defined by $\mE^{\mathrm{r}}$ by applying \Cref{alg : essential to permutation} to $E^{\mathrm{r}}= \{ (r,(i,j)) \mid (r,[i,j+i-1])\in \mE^{\mathrm{r}} \}$.
\end{remark}

\section{Realization spaces of positroids} \label{section : positroid cells}
We now turn our attention to the realization space of a positroid in the totally non-negative Grassmannian and the Grassmannian.
Recall that, given a field $\mathbb{K}$, the Grassmannian $\Gr(k, n)$ is the space of all $k$-dimensional linear subspaces of $\mathbb{K}^n$. Any point $V$ in $\Gr(k, n)$ can be represented by a $k\times n$ matrix with entries in $\mathbb{K}$.
Let $M=(m_{ij})$ be a $k\times n$ matrix of indeterminates. For a subset $I = \{i_1,\ldots,i_k\} \in \binom{[n]}{k}$, let $M_I$ denote the $k\times k$ submatrix of $M$ with the column indices $i_1,\ldots,i_k$.
The \emph{Pl\"ucker coordinates} of $V$ are $p_I(V) = \text{det}(M_I)$ for each $I \in\binom{[n]}{k}$. The Pl\"ucker coordinates do not depend on the choice of matrix $M$, up to simultaneous rescaling by a non-zero constant, and they determine the \emph{Pl\"ucker embedding} of $\Gr(k,n)$ into $\mathbb{P}^{\binom{n}{k}-1}$.
Moreover, any point $V\in\Gr(k,n)$ can be represented as $\text{span}\lbrace v_1,\ldots,v_k\rbrace$ for some $\mathbb{K}$-vector space basis $\{v_1,\dots,v_k\}$.

\subsection{Positroid cells}

Let $\mathbb{K}=\mathbb{R}$. The \emph{totally non-negative Grassmannian} $\Gr_{\geq 0}(k,n)$ is the subset of $\Gr(k,n)$ consisting of points whose {non-zero} Pl\"ucker coordinates have all the same sign. Fix a positroid $\mP$ on $[n]$ of rank $k$, defined via its basis $\mB$. Then the \emph{positroid cell} of $\mP$ is the set:
\[ \mS_{\mP} = \{ M\in \Gr_{\geq 0}(k,n) \mid M_I=0 \text{ if and only if } I\not\in \mB\}.
\]
Each positroid cell $\mS_{\mP}$ is a topological cell \cite{postnikov2006total}, and moreover, the positroid cells of $\Gr_{\geq 0}(k,n)$ glue together to form a CW-decomposition of $\Gr_{\geq 0}(k,n)$, as shown in~\cite{postnikov2009matching}. 

\medskip
In the remainder of this section, we study positroid cells. In particular, we are interested in computing the dimension of the cell corresponding to a given positroid. Combinatorial objects in bijection with positroids offer a way to compute the dimension of such cells. Specifically, it is possible to compute the dimension of a positroid cell using bounded affine permutations, as outlined in the following result.

\begin{theorem}[\protect{\cite[Theorem~3.16]{knutson2013positroid}}]\label{thm : codim = length}
    Let $\mP$ be a rank $k$ positroid on $[n]$ and let $\pi$ be the associated bounded affine permutation. The positroid cell $\mS_{\mP}$ has codimension $\ell(\pi)$ inside 
    $\Gr_{\geq 0}(k,n)$, where
    \[ \ell(\pi) = |\{ (i,j)\mid i\in [n], i<j\leq i+n \text{ and } \pi(i)>\pi(j) \}| \]
    is called \emph{length} of $\pi$ and it is equal to the number of inversions in $\pi$.
\end{theorem}

The codimension of a positroid cell can be determined from its ranked essential family or its core.

\begin{proposition}\label{thm : dim formula}
    Let $\mE^{\mathrm{r}}$ be the set of ranked essential sets of a rank $k$ positroid $\mP$. Then
    \[ \codim(\mS_{\mP}) = \sum_{(r,[i,j])\in \mE^{\mathrm{r}}} (k-r)e_{[i,j]} = \sum_{(r,[i,j])\in \mC^{\mathrm{r}}} (k-r)e_{[i,j]}. \]
\end{proposition}
\begin{proof}
    Let $\pi$ be the bounded affine permutation corresponding to the positroid $\mP$. By \Cref{thm : codim = length}, we need to prove that $\ell(\pi) = \sum_{(r,[i,j])\in \mE^{\mathrm{r}}} e_{[i,j]}(k-r) $. Consider the dotted array of $\pi$ and note that the length of $\pi$ can be computed by looking at the array as:
    \[ \ell(\pi) = \sum_{(i,j)\in D(\pi)} |D(\pi) \cap T_{(i,j)}|. \]
    Indeed, every time a dot $(i,\pi(i)-i+1)$ is counted in the previous formula,  there exists a $j<i$ such that $\pi(j)>\pi(i)$, indicating one inversion.
    
    Note that, if the dot in row $i$ lies in $D(\pi)\cap T_{(j,\pi(j)-j+1)}$, then $(j,\pi(j)-j+1)$ defines a bounded section of the diagram of $\pi$, which is non-empty since it contains the dot $(i, \pi(i)-i+1)$. Hence there is a corner, i.e. an essential set inside $P_{(j,\pi(j)-j+1)}$. Let $(r,I)\in \mE^{\mathrm{r}}$ be such an essential set. The number of dots in the same connected component in the diagram as $I$ is exactly the excess of $I$, $e_I$.
    We have to count the dot in the same connected component as the essential set $[i,j]$ as many times as the number of inversions it is contained in, i.e. as many times as the number of $h$ such that $(i,j-i+1) \in T_{(h,\pi(h)-h+1}$. Consider the antidiagonal through the point $(i,j-i+1)$. On the right side of such a diagonal there are $r$ dots, equal to the rank of $[1,n]$. Among these points, those that satisfy the property described above are the dots not in $P_{(i,j-i+1)}$. It follows that, if the essential set $I$ has rank $r$, there are $k-r$ such points.
    Hence we can write:
    \begin{align*}
        \ell(\pi) &= \sum_{(i,j)\in D(\pi)} |D(\pi) \cap T_{(i,j)}| 
        = \sum_{i\in [n]}\sum_{\mE\ni I\subseteq [i,\pi(i)]} e_I 
        = \sum_{(r,I)\in \mE^{\mathrm{r}}} (k-r)e_I.
    \end{align*}
The last equality in the statement holds because $e_I > 0$ if and only if $(r,I) \in \mC^{\mathrm{r}}$, i.e. it is in the core of the ranked essential family.
\end{proof}

\begin{example}\leavevmode
\begin{itemize}
    \item The uniform matroid $\mathcal{U}_{k,n}$ of rank $k$ on $n$ elements has, as shown in \Cref{ex : ess set uniform matroid}, essential family given by $\{ (k, [n]) \}$. Hence, the codimension of the corresponding cell is $0$. This is indeed the top-dimensional positroid cell in the totally non-negative Grassmannian.

    \item Let $\mP$ be the positroid from \Cref{ex : diagram}. Its ranked essential family is 
    \[\mE^{\mathrm{r}} = \{ (1,[5,6]), (2, [1,4]), (2, [4,7]), (3, [1,8]) \}.\]
    The codimension of the corresponding positroid cell is:
   $\codim(\mS_\mP) = (3-1) + 2(3-2) + (3-2) = 5.$
\end{itemize}  
\end{example}

\subsection{Boundaries of a positroid cell}

Having an explicit formula for the dimension, we can study how to construct positroid cells lying in the boundary of {the Zariski closure of} a given positroid cell via essential sets. In particular, we can construct all of the codimension $1$ positroid cells.

\begin{example}\label{ex : boundaries}
 Consider the positroid $\mP$ from \Cref{ex : diagram}, whose ranked essential family is $\mE^{\mathrm{r}} = \{ (1,[5,6]), (2, [1,4]), (2, [4,7]), (3, [1,8]) \}$, and whose positroid cell has codimension $5$. Then there are $9$ positroid cells of codimension $1$ contained in {the Zariski closure of} $\mS_\mP$ whose ranked essential families are described in the following table. Each row contains respectively the rank, the cyclic interval, and the excess of the essential set.
    \begin{figure}[h]
        \centering
        \includegraphics[width=0.9\linewidth]{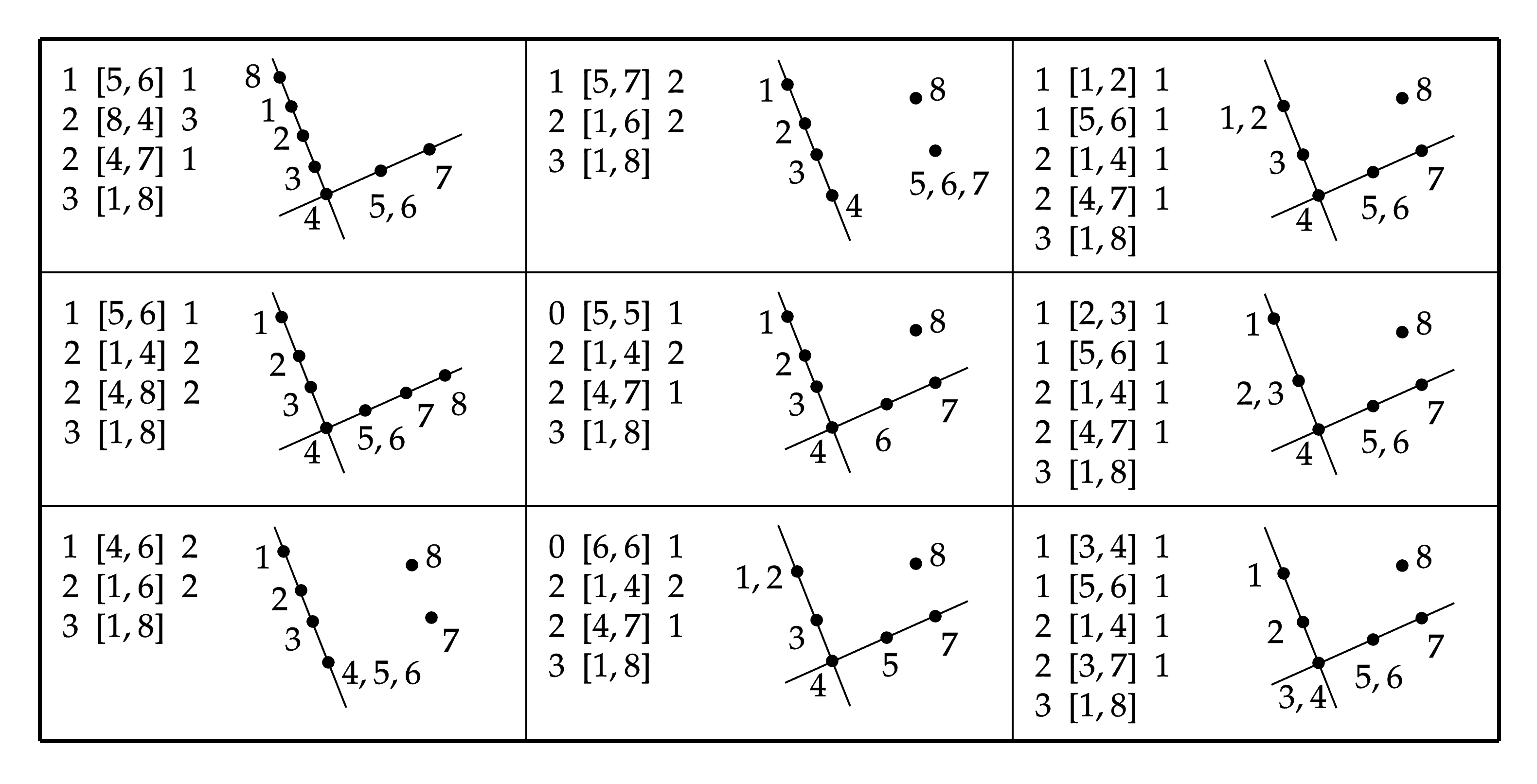}
        \caption{{The positroids whose cell lies in the Zariski closure of the cell of the positroid from \Cref{ex : boundaries}. The data on the left of each box is the rank, the cyclic interval and the excess of each essential set of the family. On the right, the point-line configuration of the positroid is depicted.}}
        \label{fig : ex codim 1 boundaries}
    \end{figure}
\end{example}

\subsection{Positroid varieties} \label{section : variety}

Given a rank $k$ positroid on $[n]$, labeled by the bounded affine permutation $\pi$, the \emph{positroid variety} $\Pi_\pi$ is the Zariski closure of the positroid cell $\mS_\pi$ inside $\Gr(k,n)$. The following result generalizes \cite[Proposition~2.3]{knutson2010puzzles}, which states an equivalent result for interval positroid varieties,~i.e. positroids defined by rank conditions on intervals.

\begin{proposition}
    The positroid variety $\Pi_\pi$ is defined as a scheme by the rank conditions $\rank(I)\leq r$ for every $(r,I)\in \mC\mE_\pi^{\mathrm{r}}$, the connected elements of the ranked essential family of $\mP_\pi$.
\end{proposition}
\begin{proof}
    By \Cref{rmk : rank from connected} and \Cref{thm : rank cyclic intervals}, the rank of each cyclic interval $[i,j]$  can be computed in terms of the rank of the connected essential sets. Hence, the rank condition for $[i,j]$ is implied by those in~$\mC\mE^{\mathrm{r}}$.
\end{proof}

\section{Small rank positroids} \label{section : small rank}

We now focus on rank $2$ positroids and show how essential sets coincide with the construction given in \cite{mohammadi2022computing}. Moreover, in this case, essential sets correspond to the circuit closures of the positroids.

\medskip
Let $\mP$ be a positroid defined by its flats $\mF$ and consider the following family:
\[ \mathcal{F}^{\mathrm{r}} = \{ (r,F) \mid F \in \mF, r =\rank(F)<|F|\}. \]

If $\mP$ is a loopless rank $2$ positroid, then $\mF$ is equal to the family of circuit closures of the positroid.

\begin{proposition}
  If $\mP$ is a loopless positroid of rank $2$,~its ranked essential family $\mE^{\mathrm{r}}$ coincides with $\mathcal{F}^{\mathrm{r}}$.
\end{proposition}
\begin{proof}
If $\mP$ has rank $2$, then $\mF^{\mathrm{r}}$ is the family of connected components of the graph associated to the positroid $\mP$ as constructed in \cite[Section 3.2]{mohammadi2022computing}. In \cite[Proposition 3.9]{mohammadi2022computing}, it is shown that every element of $\mF^{\mathrm{r}}$ is a cyclic interval, making it an essential set. Conversely, each essential set of rank $1$ gives rise to a connected component in the graph associated with the positroid $\mP$. Hence, $\mE^{\mathrm{r}} = \mF^{\mathrm{r}}$.
\end{proof}

In particular we can easily characterize positroids in terms of the family of circuit closures.

\begin{corollary}
    If $\mM$ is a loopless matroid of rank $2$, $\mM$ is a positroid if and only if $\mF^{\mathrm{r}}$ consists of cyclic intervals.
\end{corollary}

\begin{remark}
    The equivalence between essential sets and ranked circuit closures does not hold for positroids of higher rank. For instance, consider the rank $3$ positroid with ranked essential family 
    \[\mE^{\mathrm{r}} = \{{(1,[1,2]), (2,[1,5]),(2,[5,2]), (3, [1,7])} \}.\] 
    Then $\{1,5\}$ is a circuit of rank $1$ which does not appear in $\mE^{\mathrm{r}}$. {Moreover, $\{1,2\}$ is not a flat of the positroid, as its closure is given by $\{1,2,5\}$}.
\end{remark}

\bibliographystyle{abbrv}
\bibliography{Arxiv}

\bigskip
\noindent 
\textbf{Authors' addresses}

\noindent
Department of Computer Science, KU Leuven, Celestijnenlaan 200A, B-3001 Leuven, Belgium\\ 
   Department of Mathematics, KU Leuven, Celestijnenlaan 200B, B-3001 Leuven, Belgium\\
   UiT – The Arctic University of Norway, 9037 Troms\o, Norway\\
    E-mail address: {\tt fatemeh.mohammadi@kuleuven.be}

\medskip  \noindent
Department of Mathematics, KU Leuven, Celestijnenlaan 200B, B-3001 Leuven, Belgium
\\ E-mail address: {\tt francesca.zaffalon@kuleuven.be}

\end{document}